\documentclass{amsart}
\usepackage{amsfonts}
\usepackage{graphicx}
\usepackage{amscd}
\usepackage{amsmath}
\usepackage{amssymb}

\makeatletter
\@namedef{subjclassname@2010}{%
  \textup{2010} Mathematics Subject Classification}
\makeatother

\theoremstyle{plain}

\newtheorem{corollary}{\bf Corollary}

\newtheorem{lemma}{\bf Lemma}

\newtheorem{proposition}{\bf Proposition}
\newtheorem{remark}{Remark}

\newtheorem{theorem}{\bf Theorem}
\newtheorem{conjecture}{\bf Conjecture}

\theoremstyle{definition}
\newtheorem{example}{\bf Example}

\numberwithin{equation}{section}

\title[Critical metrics of the scalar and volume functionals]{Remarks on critical metrics of the scalar curvature and volume functionals on compact manifolds with boundary}

\author{H. Baltazar}
\author{E. Ribeiro Jr}

\address{Universidade Federal do Piau\'{\i} - UFPI, Departamento de Matem\'{a}tica, Campus Petr\^onio Portella, 64049-550, Te\-re\-si\-na / PI, Brazil.}
\email{halyson@ufpi.edu.br}

\address{Universidade Federal do Cear\'a - UFC, Departamento  de Matem\'atica, Campus do Pici, Av. Humberto Monte, Bloco 914,
60455-760, Fortaleza / CE, Brazil.}\email{ernani@mat.ufc.br}

\thanks{H. Baltazar and E. Ribeiro Jr were partially supported by CNPq/Brazil}
\date{December 12, 2016}

\begin{document}

\newcommand{\spacing}[1]{\renewcommand{\baselinestretch}{#1}\large\normalsize}
\spacing{1.2}

\begin{abstract}
We provide a general B\"ochner type formula which enables us to prove some rigidity results for $V$-static spaces. In particular, we show that  an  $n$-dimensional positive static triple with connected boundary and positive scalar curvature must be isometric to the standard hemisphere, provided that the metric has zero radial Weyl curvature and satis\-fies a suitable pinching condition. Moreover, we classify $V$-static spaces with non-negative sectional curvature.
\end{abstract} 

\maketitle

\section{Introduction}
\label{intro}
Let $(M^{n},\,g)$ be a connected Riemannian manifold. Following the terminology used by Miao and Tam \cite{miaotam} as well as Corvino, Eichmair and Miao \cite{CEM}, we say that $g$ is a {\it $V$-static metric} if there is a smooth function $f$ on $M^n$ and a constant $\kappa$ satisfying the $V$-static equation
\begin{equation}
\label{eqVstatic} \mathfrak{L}_{g}^{*}(f)=-(\Delta f)g+Hess\, f-fRic=\kappa g,
\end{equation} where $\mathfrak{L}_{g}^{*}$ is the formal $L^{2}$-adjoint of the linearization of the scalar curvature operator $\mathfrak{L}_{g},$ which plays an important role in problems related to prescribing the scalar curvature function. Here, $Ric,$ $\Delta$ and $Hess$ stand, respectively, for the Ricci tensor, the Laplacian operator and the Hessian form on $M^n.$ Such a function $f$ is called {\it $V$-static potential}.

It is well-known that $V$-static metrics are important in understanding the interplay between volume and scalar curvature. It arises from the modified problem of finding stationary points for the volume functional on the space of metrics whose scalar curvature is equal to a given constant (cf. \cite{CEM,miaotam,miaotamTAMS,yuan}). In general, the scalar curvature is not sufficient in controlling the volume. However,  Miao and Tam \cite{miaotamCAG} proved a rigidity result for upper hemisphere with respect to non-decreasing scalar curvature and volume. While Corvino, Eichmair and Miao  \cite{CEM} were able to show that when the metric $g$ does not admit non-trivial solution to Eq. (\ref{eqVstatic}), then one can achieve simultaneously a prescribed perturbation of the scalar curvature that is compactly supported in a bounded domain $\Omega$ and a prescribed perturbation of the volume by a small deformation of the metric in $\overline{\Omega}.$ We highlight that a Riemannian metric $g$ for which there exists a nontrivial function $f$ satisfying (\ref{eqVstatic}) must have constant scalar curvature $R$ (cf. Proposition 2.1 in \cite{CEM} and Theorem 7 in \cite{miaotam}).

The case where $\kappa\neq 0$ in (\ref{eqVstatic}) and the potential function $f$ vanishes on the boundary was studied by Miao and Tam \cite{miaotam}. In this approach, a {\it Miao-Tam critical metric} is a 3-tuple $(M^n,\,g,\,f),$ where $(M^{n},\,g)$ is a compact Riemannian manifold of dimension at least three with a smooth boundary $\partial M$ and $f: M^{n}\to \Bbb{R}$ is a smooth function such that $f^{-1}(0)=\partial M$ satisfying the overdetermined-elliptic system
\begin{equation}
\label{eqMiaoTam1} \mathfrak{L}_{g}^{*}(f)=-(\Delta f)g+Hess\, f-fRic=g.
\end{equation} Miao and Tam \cite{miaotam} showed that these critical metrics arise as critical points of the volume functional on $M^n$ when restricted to the class of metrics $g$ with prescribed constant scalar curvature such that $g_{|_{T \partial M}}=h$ for a prescribed Riemannian metric $h$ on the boundary. Some explicit examples of Miao-Tam critical metrics are in the form of warped products and those examples include the spatial Schwarzschild metrics and AdS-Schwarzschild metrics restricted to certain domains containing their horizon and bounded by two spherically symmetric spheres (cf. Corollaries 3.1 and 3.2 in \cite{miaotamTAMS}). For more details see, for instance, \cite{baltazar,BDR,BDRR,CEM,miaotam,miaotamTAMS} and \cite{yuan}.

We also remark that Eq. (\ref{eqVstatic}) can be seen as a generalization of the {\it static equation} $\mathfrak{L}_{g}^{*}(f)=0$ (cf. \cite{Lucas} and  \cite{Corvino}), namely,  $\kappa=0$ in Eq. (\ref{eqVstatic}). We remember that a {\it positive static triple} is a triple $(M^{n},\,g,\,f)$ consisting of a connected $n$-dimensional smooth manifold $M$ with boundary $\partial M$ (possibly empty), a complete Riemannian metric $g$ on $M$ and a nontrivial static potential $f\in C^{\infty}(M)$ that is non-negative, vanishes precisely on $\partial M$ and satisfies the static equation 
\begin{equation}
\label{staticequation} \mathfrak{L}_{g}^{*}(f)=-(\Delta f)g+Hess\, f-fRic=0.
\end{equation} For the sake of completeness, it is very important to recall the following classical example of positive static triple with non-empty boundary.

\begin{example}
\label{ex1}
An example of positive static triple with connected non-empty boundary is given by choosing $(\Bbb{S}^{n}_{+}(r),\,g)$ the open upper $n$-hemisphere $\Bbb{S}^{n}_{+}(r)$ of radius $r$ in $\Bbb{R}^{n+1}$ endowed with the Euclidean metric $g.$ Hence, $\partial M=\Bbb{S}^{n-1}(r)$ is the equator and the corresponding height function $f$ on $\Bbb{S}^{n}_{+}(r)$ is positive, vanishes along $\partial M=\Bbb{S}^{n-1}(r)$ and satisfies Eq. (\ref{staticequation}).
\end{example}

It has been conjectured in 1984 that the only static vacuum spacetime with positive cosmological constant and connected event horizon is the  de Sitter space of radius $r.$ This conjecture is so-called {\it cosmic no-hair conjecture} and it was formulated by Boucher, Gibbons and Horowitz in \cite{BGH}. It is closely related to Fischer-Marsden conjecture (cf. \cite{Shen}). It should be emphasized that there are positive static triple with double boundary, as for instance, the {\it Nariai space}. Hence, connectedness of the boundary is essential for Conjecture \ref{conjA} to be true. In other words, the quoted authors proposed the following conjecture.

\begin{conjecture}[Cosmic no-hair conjecture, \cite{BGH}] 
\label{conjA}
Example \ref{ex1} is the only possible $n$-dimensional positive static triple with single-horizon (connected) and positive scalar curvature.
\end{conjecture}

In the last decades some partial answers to Conjecture \ref{conjA} were achieved. For instance, if $(M^{n},\,g)$ is Einstein it suffices to apply the Obata type theorem due to Reilly \cite{Reilly2} (see also \cite{obata}) to conclude that Conjecture \ref{conjA} is true. Moreover, Kobayashi \cite{kobayashi} and Lafontaine \cite{lafontaine} proved independently that such a conjecture is also true under conformally flat condition.

For what follows, we recall that the Bach tensor on a Riemannian manifold $(M^n,g),$ $n\geq 4,$ is defined in terms of the components of the Weyl tensor $W_{ikjl}$ as follows
\begin{equation}
\label{bach} B_{ij}=\frac{1}{n-3}\nabla^{k}\nabla^{l}W_{ikjl}+\frac{1}{n-2}R_{kl}W_{i}\,^{k}\,_{j}\,^{l},
\end{equation} while for $n=3$ it is given by
\begin{equation}
\label{bach3} B_{ij}=\nabla^kC_{kij},
\end{equation} where $C_{ijk}$ stands for the Cotton tensor. We say that $(M^n,g)$ is Bach-flat when $B_{ij}=0.$

Qing and Yuan \cite{jiewei} obtained a classification result for static spaces under Bach-flat assumption. In particular, it is not hard to see that the method used by Qing and Yuan implies that such a conjecture is also true under Bach-flat assumption (cf. Theorem \ref{classsifStatic} below). In the work \cite{GHP}, Gibbons, Hartnoll and Pope constructed counterexamples to the cosmic no-hair conjecture in the cases $4\leq n\leq 8.$ However, it remains interesting to show under which conditions such a conjecture remains true. For more details on this subject and further partial answers see, for instance, \cite{Lucas,BGH,Ch,HMR,Shen} and references therein. Next, let us recall the following useful classification. 

\begin{theorem}[Kobayashi \cite{kobayashi}, Lafontaine \cite{lafontaine}, Qing and Yuan \cite{jiewei}]
\label{classsifStatic} Let $(M^{n},\,g,\,f)$ be an $n$-dimensional positive static triple with scalar curvature $R=n(n-1).$ Suppose that $(M^{n},\,g)$ is Bach-flat, then $(M^{n},\,g,\,f)$ is covered by a static triple equivalent to one of the following static triples:

\begin{enumerate}
\item The standard hemisphere with canonical metric
$$(\mathbb{S}^{n}_{+},g_{\mathbb{S}^{n-1}},f=x_{n+1}).$$
\item The standard cylinder over $\mathbb{S}^{n-1}$ with the product metric
$$\Big(M=\Big[0,\frac{\pi}{\sqrt{n}}\Big]\times\mathbb{S}^{n-1},\; g=dt^{2}+\frac{n-2}{n}g_{\mathbb{S}^{n-1}},\; f(t)=sen(\sqrt{n}t)\Big).$$
\item For some constant $m\in\Big(0,\sqrt{\frac{(n-2)^{n-2}}{n^{n}}}\Big)$ we consider the Schwarzschild space defined by 
$$\Big(M=\Big[r_{1},r_{2}]\times\mathbb{S}^{n-1},\; g=\frac{1}{1-2mt^{2-n}-t^{2}}dt^{2}+t^{2}g_{\mathbb{S}^{n-1}},\; f(t)=\sqrt{1-2mt^{2-n}-t^{2}}\Big),$$
where $r_{1}<r_{2}$ are the positive zeroes of $f.$
\end{enumerate}

\end{theorem}

In the work \cite{Lucas}, Ambrozio obtained interesting classification results for static three-dimensional manifolds with positive scalar curvature. To do so, he proved a B\"ochner type formula for three-dimensional positive static triples in\-vol\-ving the traceless Ricci tensor and the Cotton tensor. A similar B\"ochner type formula was obtained by Batista et al. \cite{BDRR} for three-dimensional Riemannian manifolds satisfying (\ref{eqMiaoTam1}). Those formulae may be used to rule out some possible new examples. In this article, we extend such B\"ochner type formulae for a more general class of metrics and arbitrary dimension $n>2.$ To be precise, we have established the following result.

\begin{theorem}\label{thmMainA}
Let $(M^{n},\,g,\,f,\kappa)$ be a connected, smooth Riemannian manifold and $f$ is a smooth function on $M^n$ satisfying the $V$-static equation (\ref{eqVstatic}). Then we have:
\begin{eqnarray}
\frac{1}{2}{\rm div}(f\nabla|Ric|^{2})&=&\Big(\frac{n-2}{n-1}|C_{ijk}|^{2}+|\nabla Ric|^{2}\Big)f+ \frac{n\kappa}{n-1}|\mathring{Ric}|^{2}\nonumber\\
&&+\Big(\frac{2}{n-1}R|\mathring{Ric}|^{2}+\frac{2n}{n-2}tr(\mathring{Ric}^{3})\Big)f\nonumber\\
&&-\frac{n-2}{n-1}W_{ijkl}\nabla_{l}fC_{ijk}-2fW_{ijkl}R_{ik}R_{jl},
\end{eqnarray} where $C$ stands for the Cotton tensor, $W$ is the Weyl tensor and $\mathring{Ric}$ is the traceless Ricci tensor.
\end{theorem}

Remembering that three-dimensional Riemannian manifolds have vanish Weyl tensor it is easy to see that Theorem \ref{thmMainA} is a generalization, for any dimension, of Theorem 3 in \cite{BDRR} as well as Proposition 12 in \cite{Lucas}.

Before presenting a couple of applications of the above formula it is fundamental to remember that a Riemannian manifold $(M^{n},\,g)$ has {\it zero radial Weyl curvature} when
\begin{equation}\label{Wradialflat}
W(\,\cdot\,,\,\cdot\,,\,\cdot\,,\nabla f)=0,
\end{equation} for a suitable potential function $f$ on $M^n.$ This class of manifolds includes the case of locally conformally flat manifolds. Moreover, this condition have been used to classify gradient Ricci solitons as well as quasi-Einstein manifolds (cf. \cite{catino,PW} and \cite{pwylie}). Here, we shall use this condition to obtain the following corollary.

\begin{corollary}\label{corA}
Let $(M^{n},\,g,\,f)$ be a compact, oriented, connected Miao-Tam critical metric with positive scalar curvature and nonnegative potential function $f.$ Suppose that:
\begin{itemize}
\item $M^n$ has zero radial Weyl curvature and
\item $|\mathring{Ric}|^{2}\leq\frac{R^{2}}{n(n-1)}.$
\end{itemize}  Then $M^n$ must be isometric to a geodesic ball in $\mathbb{S}^n.$
\end{corollary} 

It is not difficult to see that the above result generalizes Corollary 1 in \cite{BDRR}. Next, we get the following result for static spaces.

\begin{corollary}\label{corB}
Let $(M^{n},\,g,\,f)$ be a compact, oriented, connected positive static triple with positive scalar curvature. Suppose that:
\begin{itemize}
\item $M^n$ has zero radial Weyl curvature and
\item $|\mathring{Ric}|^{2}\leq\frac{R^{2}}{n(n-1)}.$
\end{itemize} Then one of the following assertions holds:

\begin{enumerate}
\item $M^n$ is equivalent to the standard hemisphere of $\Bbb{S}^n;$ or
\item $|\mathring{Ric}|^{2}=\frac{R^{2}}{{n(n-1)}}$ and $(M^{n},\,g,\,f)$ is covered by a static triple that is
equivalent to the standard cylinder.
\end{enumerate}
\end{corollary}

\begin{remark}
It is worthwhile to remark that Corollary \ref{corB} can be seen as a partial answer to Conjecture \ref{conjA}.
\end{remark}

\begin{remark}
\label{rem2}
We also point out that four-dimensional $V$-static spaces with zero radial Weyl curvature must be locally conformally flat. To prove this claim it suffices to apply the same arguments used in the initial part of the proof of Theorem 2 in \cite{BDR}.
\end{remark}

In order to proceed, we recall that a classical lemma due to Berger guarantees that any two symmetric tensor $T$ on a Riemannian manifold $(M^{n},\,g)$ with non-negative sectional curvature must satisfy

\begin{equation}
\label{eqBerger}
(\nabla_{i}\nabla_{j}T_{ik}-\nabla_{j}\nabla_{i}T_{ik})T_{jk}\geq 0,
\end{equation} In fact, we have $$(\nabla_{i}\nabla_{j}T_{ik}-\nabla_{j}\nabla_{i}T_{ik})T_{jk}=\sum_{i<j}R_{ijij}(\lambda_{i}-\lambda_{j})^{2},$$ where $\lambda_{i's}$ are the eigenvalues of tensor $T$ (cf. Lemma 4.1 in \cite{XiaCao}). Here, we shall use these data jointly with Theorem \ref{thmMainA} to deduce a rigidity result for three-dimensional Miao-Tam critical metrics with non-negative sectional curvature (see also Proposition \ref{thmKl} in Section \ref{SecK} for a version in arbitrary dimension). More precisely, we have established the following result.

\begin{theorem}\label{thmK}
Let $(M^3,\,g,\,f)$ be a three-dimensional compact, oriented, connected Miao-Tam critical metric with smooth boundary $\partial M$ and non-negative sectional curvature, $f$ is also assumed to be nonnegative. Then $M^3$ is isometric to a geodesic ball in a simply connected space form $\Bbb{R}^3$ or $\Bbb{S}^3.$
\end{theorem}

Finally, we get the following result for positive static triple.

\begin{theorem}
\label{thmstaticKA}
Let $(M^{n},\,g,\,f)$ be a positive static triple with non-negative sectional curvature, zero radial Weyl curvature and scalar curvature $R=n(n-1).$ Then up to a finite quotient $M^{n}$ is isometric to either the standard hemisphere $\Bbb{S}^{n}_{+}$ or the standard cylinder over $\mathbb{S}^{n-1}$ with the product metric described in Theorem \ref{classsifStatic}.
\end{theorem}

\section{Preliminaries}
\label{Preliminaries}

In this section we shall present some preliminaries which will be useful for the establishment of the desired results. Firstly, we remember that a $V$-static space is a Riemannian manifold $(M^{n},\,g)$ with a non-trivial solution $(f,\kappa)$ satisfying the overdetermined-elliptic system
$$-(\Delta f)g+Hessf-fRic=\kappa g,$$ where $\kappa$ is a constant. As usual, we rewrite in the tensorial language as 

\begin{equation}
\label{fundeqtensV}
-(\Delta f)g_{ij}+\nabla_{i}\nabla_{j}f-fR_{ij}=\kappa g_{ij}.
\end{equation} Tracing (\ref{fundeqtensV}) we deduce that $f$ also satisfies the equation

\begin{equation}
\label{eqtraceV}
\Delta f+\frac{R}{n-1}f+\frac{n\kappa}{n-1}=0.
\end{equation} 

 Moreover, by using (\ref{eqtraceV}) it is not difficult to check that
\begin{equation}
\label{IdRicHess} f\mathring{Ric}=\mathring{Hess f},
\end{equation} where $\mathring{T}$ stands for the traceless of $T.$

Before proceeding we recall two special tensors in the study of curvature for a Riemannian manifold $(M^n,\,g),\,n\ge 3.$  The first one is the Weyl tensor $W$ which is defined by the following decomposition formula
\begin{eqnarray}
\label{weyl}
R_{ijkl}&=&W_{ijkl}+\frac{1}{n-2}\big(R_{ik}g_{jl}+R_{jl}g_{ik}-R_{il}g_{jk}-R_{jk}g_{il}\big) \nonumber\\
 &&-\frac{R}{(n-1)(n-2)}\big(g_{jl}g_{ik}-g_{il}g_{jk}\big),
\end{eqnarray}
where $R_{ijkl}$ stands for the Riemann curvature operator $Rm,$ whereas the second one is the Cotton tensor $C$ given by
\begin{equation}
\label{cotton} \displaystyle{C_{ijk}=\nabla_{i}R_{jk}-\nabla_{j}R_{ik}-\frac{1}{2(n-1)}\big(\nabla_{i}R
g_{jk}-\nabla_{j}R g_{ik}).}
\end{equation} Note that $C_{ijk}$ is skew-symmetric in the first two indices and trace-free in any two indices. These two above tensors are related as follows
\begin{equation}
\label{cottonwyel} \displaystyle{C_{ijk}=-\frac{(n-2)}{(n-3)}\nabla_{l}W_{ijkl},}
\end{equation}provided $n\ge 4.$

For our purpose we also remember that as consequence of Bianchi identity we have
\begin{equation}\label{Bianchi}
({\rm div} Rm)_{jkl}=\nabla_kR_{jl}-\nabla_lR_{jk}.
\end{equation} Moreover, from commutation formulas (Ricci
identities), for any Riemannian manifold $(M^{n},\,g)$ we have
\begin{equation}\label{idRicci}
\nabla_i\nabla_j R_{kl}-\nabla_j\nabla_i R_{kl}=R_{ijks}R_{sl}+R_{ijls}R_{ks},
\end{equation} for more details see \cite{chow,Via}. 

To conclude this section, we shall present the following lemma for $V$-static spaces, which was previously obtained in \cite{BDR} for Miao-Tam critical metrics. 

\begin{lemma}
\label{L1}
Let $(M^{n},g)$ be a connected, smooth Riemannian manifold and $f$ is a smooth function on $M^n$ satisfying Eq. (\ref{eqVstatic}). Then we have:
$$f(\nabla_{i}R_{jk}-\nabla_{j}R_{ik})=R_{ijkl}\nabla_{l}f+\frac{R}{n-1}(\nabla_{i}fg_{jk}-\nabla_{j}fg_{ik})-(\nabla_{i}fR_{jk}-\nabla_{j}f R_{ik}).$$
\end{lemma}

\begin{proof} The proof is standard, and it follows the same steps of Lemma 1 in \cite{BDR}. For sake of completeness we shall sketch it here. Firstly, since $g$ is parallel we may use  (\ref{fundeqtensV}) to infer

\begin{equation}
\label{eq1lem1}
(\nabla_{i}f)R_{jk}+f\nabla_{i}R_{jk}=\nabla_{i}\nabla_{j}\nabla_{k}f-(\nabla_{i}\Delta f)g_{jk}.
\end{equation} Next, since $M^n$ has constant scalar curvature we have from (\ref{eqtraceV}) that $$\nabla_{i}\Delta f=-\frac{R}{n-1}\nabla_{i}f,$$ which substituted into (\ref{eq1lem1}) gives 

\begin{equation}
\label{eq2lem2}
f\nabla_{i}R_{jk}=-(\nabla_{i}f)R_{jk}+\nabla_{i}\nabla_{j}\nabla_{k}f+\frac{R}{n-1}\nabla_{i}fg_{jk}.
\end{equation} Finally, we apply the Ricci identity to arrive at
\begin{eqnarray*}
f(\nabla_{i}R_{jk}-\nabla_{j}R_{ik})=R_{ijkl}\nabla_{l}f+\frac{R}{n-1}(\nabla_{i}f g_{jk}-\nabla_{j}f g_{ik})-(\nabla_{i}f R_{jk}-\nabla_{j}f R_{ik}),
\end{eqnarray*} as we wanted to prove.
\end{proof}

\section{A B\"ochner type Formula and Applications}

In this section we shall provide a general B\"ochner type formula, which enables us to prove some rigidity results for $V$-static spaces. To do so, we shall obtain some identities involving the Cotton tensor and Weyl tensor on Riemannian manifolds satisfying the $V$-static equation. Following the notation employed in \cite{BDR}, we can use (\ref{weyl}) jointly with Lemma~\ref{L1} to obtain

\begin{equation}\label{auxT}
fC_{ijk}=T_{ijk}+W_{ijkl}\nabla_{l}f,
\end{equation} where the auxiliary tensor $T_{ijk}$ is defined as

\begin{eqnarray}\label{TensorT}
T_{ijk}&=&\frac{n-1}{n-2}(R_{ik}\nabla_{j}f-R_{jk}\nabla_{i}f)+\frac{1}{n-2}(g_{ik}R_{js}\nabla_{s}f-g_{jk}R_{is}\nabla_{s}f)\nonumber\\
&&-\frac{R}{n-2}(g_{ik}\nabla_{j}f-g_{jk}\nabla_{i}f).
\end{eqnarray}

In the sequel, we obtain a divergent formula for any Riemannian manifold $(M^{n},g)$ with constant scalar curvature.

\begin{lemma}\label{lemadiv1}
Let $(M^{n},g)$ be a connected Riemannian manifold with constant scalar curvature and $f:M\rightarrow\mathbb{R}$ is a smooth function defined on $M.$ Then we have:
\begin{eqnarray*}
{\rm div}(f\nabla|Ric|^{2})&=&-f|C_{ijk}|^{2}+2f|\nabla Ric|^{2}+\langle\nabla f,\nabla|Ric|^{2}\rangle+\frac{2n}{n-2}fR_{ij}R_{ik}R_{jk}\\
&&-\frac{4n-2}{(n-1)(n-2)}fR|\mathring{Ric}|^{2}-\frac{2}{n(n-2)}fR^{3}+2\nabla_{i}(fC_{ijk}R_{jk})\\
&&+2C_{ijk}\nabla_{j}fR_{ik}-2fW_{ijkl}R_{ik}R_{jl}.
\end{eqnarray*}
\end{lemma}

\begin{proof}
Firstly, since $M^n$ has constant scalar curvature we immediately get
\begin{eqnarray*}
f|C_{ijk}|^{2}&=&f|\nabla_{i}R_{jk}-\nabla_{j}R_{ik}|^{2}\\
&=&2f|\nabla Ric|^{2}-2f\nabla_{i}R_{jk}\nabla_{j}R_{ik}.
\end{eqnarray*} On the other hand, easily one verifies that
\begin{eqnarray*}
\nabla_{j}(f\nabla_{i}R_{jk}R_{ik})&=&\nabla_{j}f\nabla_{i}R_{jk}R_{ik}+f\nabla_{j}\nabla_{i}R_{jk}R_{ik}\\&&+f\nabla_{i}R_{jk}\nabla_{j}R_{ik}.
\end{eqnarray*} Hence, it follows that
\begin{eqnarray*}
f|C_{ijk}|^{2}&=&2f|\nabla Ric|^{2}-2\nabla_{j}(f\nabla_{i}R_{jk}R_{ik})+2\nabla_{j}f\nabla_{i}R_{jk}R_{ik}\\&&+2f\nabla_{j}\nabla_{i}R_{jk}R_{ik}.
\end{eqnarray*}

Next, from commutation formula for second covariant derivative of the Ricci curvature (see Eq. (\ref{idRicci})) combined with (\ref{cotton}), we deduce

\begin{eqnarray}\label{normC}
f|C_{ijk}|^{2}&=&2f|\nabla Ric|^{2}+2\nabla_{j}f(C_{ijk}+\nabla_{j}R_{ik})R_{ik}\nonumber\\
&&+2f(R_{ij}R_{ik}R_{jl}-R_{ijkl}R_{ik}R_{jl})-2\nabla_{j}(f\nabla_{i}R_{jk}R_{ik})\nonumber\\
&=&2f|\nabla Ric|^{2}+2C_{ijk}\nabla_{j}fR_{ik}+\langle\nabla f,\nabla|Ric|^{2}\rangle\nonumber\\
&&+2f(R_{ij}R_{ik}R_{jk}-R_{ijkl}R_{ik}R_{jl})-2\nabla_{j}(f\nabla_{i}R_{jk}R_{ik}).
\end{eqnarray} Now, substituting (\ref{weyl}) into (\ref{normC}) we achieve

\begin{eqnarray*}
f|C_{ijk}|^{2}&=&2f|\nabla Ric|^{2}+2C_{ijk}\nabla_{j}fR_{ik}+\langle\nabla f,\nabla|Ric|^{2}\rangle+2fR_{ij}R_{ik}R_{jk}\\
&&-2fW_{ijkl}R_{ik}R_{jl}-\frac{2f}{n-2}(2R|Ric|^{2}-2R_{ij}R_{ik}R_{jk})\\
&&+\frac{2Rf}{(n-1)(n-2)}(R^{2}-|Ric|^{2})-2\nabla_{j}(f\nabla_{i}R_{jk}R_{ik})\\
&=&2f|\nabla Ric|^{2}+2C_{ijk}\nabla_{j}fR_{ik}+\langle\nabla f,\nabla|Ric|^{2}\rangle+\frac{2n}{n-2}fR_{ij}R_{ik}R_{jk}\\
&&-2fW_{ijkl}R_{ik}R_{jl}-\frac{(4n-2)}{(n-1)(n-2)}fR|Ric|^{2}+\frac{2}{(n-1)(n-2)}fR^{3}\\
&&-2\nabla_{j}(f\nabla_{i}R_{jk}R_{ik}),\\
\end{eqnarray*} which can be rewritten as
\begin{eqnarray*}
f|C_{ijk}|^{2}&=&2f|\nabla Ric|^{2}+2C_{ijk}\nabla_{j}fR_{ik}+\langle\nabla f,\nabla|Ric|^{2}\rangle+\frac{2n}{n-2}fR_{ij}R_{ik}R_{jk}\nonumber\\
&&-2fW_{ijkl}R_{ik}R_{jl}-\frac{(4n-2)}{(n-1)(n-2)}fR|\mathring{Ric}|^{2}-\frac{2}{n(n-2)}fR^{3}\nonumber\\
&&-2\nabla_{j}(f\nabla_{i}R_{jk}R_{ik})\nonumber\\
&=&2f|\nabla Ric|^{2}+2C_{ijk}\nabla_{j}fR_{ik}+\langle\nabla f,\nabla|Ric|^{2}\rangle+\frac{2n}{n-2}fR_{ij}R_{ik}R_{jk}\nonumber\\
&&-2fW_{ijkl}R_{ik}R_{jl}-\frac{(4n-2)}{(n-1)(n-2)}fR|\mathring{Ric}|^{2}-\frac{2}{n(n-2)}fR^{3}\nonumber\\
&&+2\nabla_{i}(fC_{ijk}R_{jk})-{\rm div}(f\nabla|Ric|^{2}),
\end{eqnarray*} where we used (\ref{cotton}) to justify the second equality. So, the proof is completed.

\end{proof}

Proceeding, we shall deduce another divergent formula, which plays a crucial role in the proof of Theorem~\ref{thmMainA}.

\begin{lemma}\label{lemadiv2}
Let $(M^{n},g,f,\kappa)$ be a $V$-static space. Then we have:
\begin{eqnarray*}
\frac{1}{2}{\rm div}(f\nabla|Ric|^{2})&=&-f|C_{ijk}|^{2}+f|\nabla Ric|^{2}+\langle\nabla f,\nabla|Ric|^{2}\rangle-\frac{n\kappa}{n-1}|\mathring{Ric}|^{2}\\
&&+2\nabla_{i}(fC_{ijk}R_{jk}).
\end{eqnarray*}
\end{lemma}

\begin{proof} To start with, we use Lemma \ref{L1} together with Eq. (\ref{cotton}) to infer

\begin{eqnarray*}
\nabla_{i}(\nabla_{j}fR_{ik}R_{jk}+R_{ijkl}\nabla_{l}fR_{jk})&=&\nabla_{i}(\nabla_{j}f R_{ik}R_{jk})\nonumber\\
&&+\nabla_{i}\Big[f C_{ijk}R_{jk}-\frac{R}{n-1}(\nabla_{i}f R-\nabla_{j}f R_{ji})\\
&&+(|Ric|^{2}\nabla_{i}f-\nabla_{j}fR_{ik}R_{jk})\Big].
\end{eqnarray*} Rearranging the terms we immediately deduce

\begin{eqnarray*}
\nabla_{i}(\nabla_{j}fR_{ik}R_{jk}+R_{ijkl}\nabla_{l}fR_{jk})&=&\nabla_{i}(f C_{ijk}R_{jk})\nonumber\\
&&+\nabla_{i}\Big[-\frac{R^{2}}{n-1}\nabla_{i}f +\frac{R}{n-1} R_{ji}\nabla_{j}f+|Ric|^{2}\nabla_{i}f\Big],
\end{eqnarray*} and remembering that $(M^{n},g)$ has constant scalar curvature we use the twice contracted second Bianchi identity ($2{\rm div} Ric=\nabla R=0)$ to get

\begin{eqnarray}
\label{lmn}
\nabla_{i}(\nabla_{j}fR_{ik}R_{jk}+R_{ijkl}\nabla_{l}fR_{jk})&=&\nabla_{i}(f C_{ijk}R_{jk})-\frac{R^{2}}{n-1}\Delta f+\frac{R}{n-1}\nabla_{i}\nabla_{j}f R_{ji}\nonumber\\
&&+|Ric|^{2}\Delta f+\langle\nabla f,\nabla|Ric|^{2}\rangle.
\end{eqnarray} Therefore, substituting (\ref{fundeqtensV}) and (\ref{eqtraceV}) into (\ref{lmn}) we obtain

\begin{eqnarray*}
\nabla_{i}(\nabla_{j}fR_{ik}R_{jk}+R_{ijkl}\nabla_{l}fR_{jk})&=&\nabla_{i}(f C_{ijk}R_{jk})+\langle\nabla f,\nabla|Ric|^{2}\rangle-\frac{R^{2}}{n-1}\Delta f\nonumber\\
&&+\frac{R}{n-1}(f R_{ij}+(\Delta f+\kappa)g_{ij})R_{ji}+\Delta f|Ric|^{2}\nonumber\\
&=&\nabla_{i}(f C_{ijk}R_{jk})+\langle\nabla f,\nabla|Ric|^{2}\rangle\nonumber\\
&&+\frac{R}{n-1}f|Ric|^{2}+\frac{R^{2}\kappa}{n-1} +\frac{-Rf-n\kappa}{n-1}|Ric|^{2}.
\end{eqnarray*} From this, it follows that

\begin{eqnarray}\label{auxdiv1}
\nabla_{i}(\nabla_{j}fR_{ik}R_{jk}+R_{ijkl}\nabla_{l}fR_{jk})&=&\nabla_{i}(f C_{ijk}R_{jk})+\langle\nabla f,\nabla|Ric|^{2}\rangle\nonumber\\&&-\frac{n\kappa}{n-1}|\mathring{Ric}|^{2}.
\end{eqnarray} At the same time, notice that

\begin{eqnarray*}
\nabla_{i}(\nabla_{j}fR_{ik}R_{jk}+R_{ijkl}\nabla_{l}fR_{jk})&=&\nabla_{i}\nabla_{j}fR_{ik}R_{jk}+\nabla_{j}fR_{ik}\nabla_{i}R_{jk}+\nabla_{i}R_{ijkl}\nabla_{l}fR_{jk}\\
&&+R_{ijkl}\nabla_{i}\nabla_{l}fR_{jk}+R_{ijkl}\nabla_{l}f\nabla_{i}R_{jk}.
\end{eqnarray*} Hence, it follows from Lemma~\ref{L1} and (\ref{fundeqtensV}) that

\begin{eqnarray*}
\nabla_{i}(\nabla_{j}fR_{ik}R_{jk}+R_{ijkl}\nabla_{l}fR_{jk})&=&f(R_{ij}R_{ik}R_{jk}-R_{ijkl}R_{ik}R_{jl})+\nabla_{j}fR_{ik}\nabla_{i}R_{jk}\\
&&+C_{ijk}\nabla_{j}fR_{ik}+fC_{ijk}\nabla_{i}R_{jk}\\
&&+(\nabla_{i}fR_{jk}-\nabla_{j}fR_{ik})\nabla_{i}R_{jk}\\
&=&f(R_{ij}R_{ik}R_{jk}-R_{ijkl}R_{ik}R_{jl})+C_{ijk}\nabla_{j}fR_{ik}\\
&&+fC_{ijk}\nabla_{i}R_{jk}+\frac{1}{2}\langle\nabla f,\nabla|Ric|^{2}\rangle.\\
\end{eqnarray*}
Proceeding, we use that the Cotton tensor is skew-symmetric in the first two indices and Eq. (\ref{normC}) to infer

\begin{eqnarray}\label{auxdiv2}
\nabla_{i}(\nabla_{j}fR_{ik}R_{jk}+R_{ijkl}\nabla_{l}fR_{jk})&=&f|C_{ijk}|^{2}-f|\nabla Ric|^{2}+\nabla_{j}(f\nabla_{i}R_{jk}R_{ik})\nonumber\\
&=&f|C_{ijk}|^{2}-f|\nabla Ric|^{2}-\nabla_{i}(fC_{ijk}R_{jk})\nonumber\\
&&+\frac{1}{2}{\rm div}(f\nabla|Ric|^{2}).
\end{eqnarray} Finally, it suffices to compare (\ref{auxdiv1}) and (\ref{auxdiv2}) to get the desired result.
\end{proof}

\subsection{Proof of Theorem~\ref{thmMainA}}
\begin{proof}

First of all, a  straightforward computation using (\ref{auxT}) as well as (\ref{TensorT}) allows us to deduce

\begin{eqnarray*}
f|C_{ijk}|^{2}=\frac{2(n-1)}{(n-2)}R_{ik}\nabla_{j}fC_{ijk}+W_{ijkl}\nabla_{l}fC_{ijk},
\end{eqnarray*} where we have used that $C_{ijk}$ is skew-symmetric in the first two indices and trace-free in any two indices. Whence, substituting this data into Lemma~\ref{lemadiv1} we obtain the following expression

\begin{eqnarray}\label{eqaux}
{\rm div}(f\nabla|Ric|^{2})&=&2f|\nabla Ric|^{2}+\langle\nabla f,\nabla|Ric|^{2}\rangle+\frac{2n}{(n-2)}fR_{ij}R_{ik}R_{jk}-\frac{1}{(n-1)}f|C_{ijk}|^{2}\nonumber\\
&&-\frac{4n-2}{(n-1)(n-2)}fR|\mathring{Ric}|^{2}-\frac{2}{n(n-2)}fR^{3}+2\nabla_{i}(fC_{ijk}R_{jk})\nonumber\\
&&-\frac{(n-2)}{(n-1)}W_{ijkl}\nabla_{l}fC_{ijk}-2fW_{ijkl}R_{ik}R_{jl}.
\end{eqnarray}

Now, comparing the expression obtained in (\ref{eqaux}) with Lemma~\ref{lemadiv2} we arrive at

\begin{eqnarray}\label{auxthm}
\frac{1}{2}{\rm div}(f\nabla|Ric|^{2})&=&\Big(\frac{n-2}{n-1}|C_{ijk}|^{2}+|\nabla Ric|^{2}\Big)f+\frac{n\kappa}{n-1}|\mathring{Ric}|^{2}\nonumber\\
&&+\frac{2n}{n-2}fR_{ij}R_{ik}R_{jk}-\frac{4n-2}{(n-1)(n-2)}fR|\mathring{Ric}|^{2}-\frac{2}{n(n-2)}fR^{3}\nonumber\\
&&-\frac{n-2}{n-1}W_{ijkl}\nabla_{l}fC_{ijk}-2fW_{ijkl}R_{ik}R_{jl}.
\end{eqnarray}

On the other hand, remembering that $\mathring{R}_{ij}=R_{ij}-\frac{R^2}{n}g,$ it is not hard to check that

$$fR_{ij}R_{ik}R_{jk}=f\mathring{R}_{ij}\mathring{R}_{jk}\mathring{R}_{ik}+\frac{3}{n}fR|\mathring{Ric}|^{2}+\frac{fR^{3}}{n^{2}}.$$ This substituted into (\ref{auxthm}) gives 
\begin{eqnarray*}
\frac{1}{2}{\rm div}(f\nabla|Ric|^{2})&=&\Big(\frac{n-2}{n-1}|C_{ijk}|^{2}+|\nabla Ric|^{2}\Big)f+\frac{n\kappa}{n-1}|\mathring{Ric}|^{2}\\
&&+\Big(\frac{2}{n-1}R|\mathring{Ric}|^{2}+\frac{2n}{n-2}tr(\mathring{Ric}^{3})\Big)f\\
&&-\frac{n-2}{n-1}W_{ijkl}\nabla_{l}fC_{ijk}-2fW_{ijkl}R_{ik}R_{jl},
\end{eqnarray*} which finishes the proof of the theorem.

\end{proof}

\subsection{Proof of Corollaries \ref{corA} and  \ref{corB}}

\begin{proof}

In order to prove Corollaries \ref{corA} and \ref{corB}, we recall that the Cotton tensor and the divergence of the Weyl tensor are related as follows
\begin{equation}\label{cottonwyel}
C_{ijk}=-\frac{n-2}{n-3}\nabla_{l}W_{ijkl}.
\end{equation} Notice also that the zero radial Weyl curvature condition, namely, $W_{ijkl}\nabla_{l}f=0$, jointly with (\ref{cottonwyel}) and (\ref{fundeqtensV}) yields

\begin{eqnarray*}
0&=&\nabla_{i}(W_{ijkl}\nabla_{k}fR_{jl})\\&=&\nabla_{i}W_{ijkl}\nabla_{k}fR_{jl}+W_{ijkl}\nabla_{i}\nabla_{k}fR_{jl}\\
&=&\frac{n-3}{n-2}C_{klj}\nabla_{k}fR_{jl}+fW_{ijkl}R_{ik}R_{jl}.
\end{eqnarray*} By using that the Cotton tensor is skew-symmetric in two first indices we obtain

\begin{eqnarray*}
fW_{ijkl}R_{ik}R_{jl}&=&\frac{n-3}{2(n-2)}C_{ijk}(\nabla_{j}fR_{ik}-\nabla_{i}fR_{ik}),
\end{eqnarray*} which can be succinctly rewritten as follows
\begin{equation*}
fW_{ijkl}R_{ik}R_{jl}=\frac{n-3}{2(n-1)}C_{ijk}T_{ijk}.
\end{equation*}From this, it follows from (\ref{auxT})  that

\begin{equation}
\label{eqw12}
fW_{ijkl}R_{ik}R_{jl}=\frac{n-3}{2(n-1)}f|C_{ijk}|^{2}.
\end{equation} Now, comparing (\ref{eqw12}) with Theorem~\ref{thmMainA} we achieve

\begin{eqnarray*}
\frac{1}{2}{\rm div}(f\nabla|Ric|^{2})&=&\Big(\frac{n-2}{n-1}|C_{ijk}|^{2}+|\nabla Ric|^{2}\Big)f+\frac{n\kappa}{n-1}|\mathring{Ric}|^{2}\\
&&+\Big(\frac{2}{n-1}R|\mathring{Ric}|^{2}+\frac{2n}{n-2}tr(\mathring{Ric}^{3})\Big)f\\
&&-\frac{n-3}{n-1}f|C_{ijk}|^{2},
\end{eqnarray*} so that

\begin{eqnarray}
\label{klmn}
\frac{1}{2}{\rm div}(f\nabla|Ric|^{2})&=&\Big(\frac{1}{n-1}|C_{ijk}|^{2}+|\nabla Ric|^{2}\Big)f+\frac{n\kappa}{n-1}|\mathring{Ric}|^{2}\nonumber\\
&&+\Big(\frac{2}{n-1}R|\mathring{Ric}|^{2}+\frac{2n}{n-2} tr(\mathring{Ric}^{3})\Big)f.
\end{eqnarray}

Before proceeding it is important to remember that the classical Okumura's lemma (cf. \cite{Okumura} Lemma 2.1) guarantees

\begin{equation}
\label{okumuradesig}
tr(\mathring{Ric}^{3})\geq-\frac{n-2}{\sqrt{n(n-1)}}|\mathring{Ric}|^{3}.
\end{equation} Therefore, upon integrating Eq. (\ref{klmn}) over $M$  we use (\ref{okumuradesig}) to arrive at

\begin{eqnarray}\label{intaux}
0&\geq&\int_{M}\Big(\frac{1}{n-1}|C_{ijk}|^{2}+|\nabla Ric|^{2}\Big)f dM_{g}+\frac{n\kappa}{n-1}\int_{M}|\mathring{Ric}|^{2} dM_{g}\nonumber\\
&&+\int_{M}\frac{2n}{\sqrt{n(n-1)}}|\mathring{Ric}|^{2}\Big(\frac{R}{\sqrt{n(n-1)}}-|\mathring{Ric}|\Big)f dM_{g}.
\end{eqnarray}

We now suppose that $\kappa=1,$ that is, $(M^{n},\,g)$ is a Miao-Tam critical metric, we may use our assumption into (\ref{intaux}) to conclude that  $|\mathring{Ric}|^{2}=0$ and this forces $M^{n}$ to be Einstein. So, it suffices to apply Theorem 1.1 in \cite{miaotamTAMS} to conclude that $(M^{n},g)$ is isometric to a geodesic ball in $\mathbb{S}^{n}$ and this concludes the proof of Corollary \ref{corA}.

From now on we assume that $\kappa=0,$ that is, $(M^{n},\,g)$ is a static space. In this case, our assumption substituted into (\ref{intaux}) guarantees that either $|\mathring{Ric}|^{2}=0$ or $|\mathring{Ric}|^{2}=\frac{R^2}{n(n-1)}.$ In the first case, we conclude that $(M^{n},\,g)$ is an Einstein manifold. Then, it suffices to apply Lemma 3 in \cite{Reilly2} to conclude that $M^n$ is isometric to a hemisphere of $\Bbb{S}^n.$ In the second one, notice that $M^n$ must have vanish Cotton tensor and parallel Ricci curvature. From this, we can use (\ref{bach}) to infer
$$(n-2)B_{ij}=\nabla_{k}C_{kij}+W_{ikjl}R_{kl}=W_{ikjl}R_{kl},$$
and consequently, by using the static equation, we deduce

\begin{eqnarray*}
(n-2)fB_{ij}&=&W_{ikjl}\nabla_{k}\nabla_{l}f\\
&=&\nabla_{k}(W_{ijkl}\nabla_{l}f)-\nabla_{k}W_{ikjl}\nabla_{l}f.
\end{eqnarray*} Hence, our assumption on Weyl curvature tensor jointly with (\ref{cottonwyel}) yields

\begin{eqnarray*}
(n-2)f B_{ij}&=&-\nabla_{k}W_{jlik}\nabla_{l}f\\
&=&\frac{n-3}{n-2}C_{jli}\nabla_{l}f=0.
\end{eqnarray*} From here it follows that $(M^{n},g)$ is Bach-flat. Hence, the result follows from Theorem \ref{thmstaticKp} (see also  Theorem 1 in \cite{Lucas} for $n=3$). This is what we wanted to prove.

\end{proof}

\section{Critical metrics with non-negative sectional curvature}
\label{SecK}

In the last decades there have been a lot of interesting results concerning the geometry of manifolds with non-negative sectional curvature. In this context, as it was previously mentionded any two symmetric tensor $T$ on a Riemannian manifold $(M^{n},\,g)$ with non-negative sectional curvature must satisfy

\begin{equation}
\label{eqBerger}
(\nabla_{i}\nabla_{j}T_{ik}-\nabla_{j}\nabla_{i}T_{ik})T_{jk}\geq 0.
\end{equation} In fact, we have $$(\nabla_{i}\nabla_{j}T_{ik}-\nabla_{j}\nabla_{i}T_{ik})T_{jk}=\sum_{i<j}R_{ijij}(\lambda_{i}-\lambda_{j})^{2},$$ where $\lambda_{i's}$ are the eigenvalues of tensor $T$ (cf. Lemma 4.1 in \cite{XiaCao}). In particular, choosing $T=Ric$ we immediately get

\begin{equation*}
(\nabla_{i}\nabla_{j}R_{ik}-\nabla_{j}\nabla_{i}R_{ik})R_{jk}\geq 0.
\end{equation*} This combined with (\ref{idRicci}) yields
\begin{equation}
\label{mnb}
R_{ij}R_{jk}R_{ik}-R_{ijkl}R_{jl}R_{ik}\geq 0.
\end{equation}

In the sequel, we shall deduce an useful expression to $R_{ij}R_{jk}R_{ik}-R_{ijkl}R_{jl}R_{ik}$ on any Riemannian manifold.

\begin{lemma}
\label{lemmaK1}
Let $(M^{n},\,g)$ be a Riemannian manifold. Then we have:
$$R_{ij}R_{jk}R_{ik}-R_{ijkl}R_{jl}R_{ik}=\frac{1}{n-1}R|\mathring{Ric}|^{2}+\frac{n}{n-2}tr(\mathring{Ric}^{3})-W_{ijkl}R_{ik}R_{jl}.$$
\end{lemma}
\begin{proof} By using the definition of the Riemann tensor (\ref{weyl}) we obtain

\begin{eqnarray*}
R_{ij}R_{jk}R_{ik}-R_{ijkl}R_{jl}R_{ik}&=&\frac{n}{n-2}R_{ij}R_{jk}R_{ik}-W_{ijkl}R_{jl}R_{ik}\nonumber\\&&-\frac{(2n-1)}{(n-1)(n-2)}R|Ric|^{2}+\frac{R^{3}}{(n-1)(n-2)}
\end{eqnarray*} so that

\begin{eqnarray}
\label{12fg}
R_{ij}R_{jk}R_{ik}-R_{ijkl}R_{jl}R_{ik}&=&\frac{n}{n-2}R_{ij}R_{jk}R_{ik}-W_{ijkl}R_{jl}R_{ik}\nonumber\\&&-\frac{(2n-1)}{(n-1)(n-2)}R|\mathring{Ric}|^{2}-\frac{1}{n(n-2)}R^{3}.
\end{eqnarray} On the other hand, we already know that 

$$R_{ij}R_{ik}R_{jk}=\mathring{R}_{ij}\mathring{R}_{jk}\mathring{R}_{ik}+\frac{3}{n}R|\mathring{Ric}|^{2}+\frac{R^{3}}{n^{2}}.$$ This substituted into (\ref{12fg}) gives the desired result. 

\end{proof}

Since three-dimensional Riemannian manifolds have vanish Weyl tensor the proof of Theorem \ref{thmK} follows as an immediate consequence of the following slightly stronger result. 

\begin{proposition}\label{thmKl}
Let $(M^n,\,g,\,f)$ be a compact, oriented, connected Miao-Tam critical metric with smooth boundary $\partial M$ and non-negative sectional curvature, $f$ is also assumed to be nonnegative. Suppose that $M^n$ has zero radial Weyl curvature, then $M^n$ is isometric to a geodesic ball in a simply connected space form $\Bbb{R}^n$ or $\Bbb{S}^n.$
\end{proposition}

\begin{proof} To begin with, we multiply by $f$ the expression obtained in Lemma \ref{lemmaK1} and then we use Theorem \ref{thmMainA} to infer

\begin{eqnarray*}
\frac{1}{2}{\rm div}(f\nabla|Ric|^{2})&=&\Big(\frac{n-2}{n-1}|C_{ijk}|^{2}+|\nabla Ric|^{2}\Big)f+\frac{n}{n-1}|\mathring{Ric}|^{2}\\
&&+2\Big(R_{ij}R_{jk}R_{ik}-R_{ijkl}R_{jl}R_{ik}\Big)f\\
&&-\frac{n-2}{n-1}W_{ijkl}\nabla_{l}fC_{ijk}
\end{eqnarray*} and since $M^n$ has zero radial Weyl curvature we get
\begin{eqnarray*}
\frac{1}{2}{\rm div}(f\nabla|Ric|^{2})&=&\Big(\frac{n-2}{n-1}|C_{ijk}|^{2}+|\nabla Ric|^{2}\Big)f+\frac{n}{n-1}|\mathring{Ric}|^{2}\nonumber\\
&&+2\Big(R_{ij}R_{jk}R_{ik}-R_{ijkl}R_{jl}R_{ik}\Big)f.
\end{eqnarray*} Finally, upon integrating the above expression over $M^n$ we use (\ref{mnb}) to conclude that $\mathring{Ric}=0$ and then $(M^{n}\,g)$ is Einstein. Hence, we apply Theorem 1.1 in \cite{miaotamTAMS} to conclude that $(M^{n},g)$ is isometric to a geodesic ball in $\mathbb{R}^{n}$ or $\mathbb{S}^{n}.$ This finishes the proof of the proposition.

\end{proof}

Proceeding, we shall prove Theorem \ref{thmstaticKA}, which was announced in Section \ref{intro}.

\begin{theorem}
Let $(M^{n},\,g,\,f)$ be a positive static triple with non-negative sectional curvature, zero radial Weyl curvature and scalar curvature $R=n(n-1).$ Then Then up to a finite quotient $M^{n}$ is isometric to either the standard hemisphere $\Bbb{S}^{n}_{+}$ or the standard cylinder over $\mathbb{S}^{n-1}$ with the product metric described in Theorem \ref{classsifStatic}.
\end{theorem}

\begin{proof} The proof looks like the one from the previous theorem. In fact, substituting Lemma \ref{lemmaK1} into Theorem \ref{thmMainA} we arrive at
\begin{eqnarray*}
\frac{1}{2}{\rm div}(f\nabla|Ric|^{2})&=&\Big(\frac{1}{n-1}|C_{ijk}|^{2}+|\nabla Ric|^{2}\Big)f\nonumber\\&&+2\Big(R_{ij}R_{jk}R_{ik}-R_{ijkl}R_{jl}R_{ik}\Big)f.
\end{eqnarray*} Whence, we integrate the above expression over $M^n$ and then use (\ref{mnb}) to conclude that $(M^{n},\,g)$ must have vanish Cotton tensor and parallel Ricci curvature. Finally, it suffices to repeat the same arguments applied in final steps of the proof of Corollary \ref{corB}. So, the proof is completed.  
\end{proof}

As an immediate consequence of the previous theorem we get the following result.

\begin{corollary}
\label{thmstaticKp}
Let $(M^{3},\,g,\,f)$ be a three-dimensional positive static triple with non-negative sectional curvature and normalized scalar curvature $R=6.$ Then up to a finite quotient $M^{3}$ is isometric to either the standard hemisphere $\Bbb{S}^{3}_{+}$ or the standard cylinder over $\mathbb{S}^{2}$ with the product metric described in Theorem \ref{classsifStatic}.
\end{corollary}

We point out that, by a different approach, Ambrozio \cite{Lucas} was able to show that a three-dimensional compact positive static triple with scalar curvature $6$ and non-negative Ricci curvature must be equivalent to the standard hemisphere or is covered by the standard cylinder.

\end{document}